\documentclass[10pt]{article}
\usepackage{epsfig}
\usepackage{amssymb,amsmath,amsthm,url}
\usepackage{multirow}
\usepackage{booktabs}
\usepackage{natbib}
\usepackage{hyperref}
\usepackage{setspace}
\usepackage{dsfont}

\usepackage{times,amsfonts,eucal}
\usepackage{graphicx,ifthen}
\usepackage{wrapfig}
\usepackage{latexsym}
\usepackage{color}
\usepackage{mathrsfs}
\usepackage{bm}
\usepackage{bbm}
\usepackage{srctex}

\usepackage{mathtools}
\DeclarePairedDelimiter{\floor}{\lfloor}{\rfloor}

\theoremstyle{plain}
\newtheorem{theorem}{Theorem}[section]
\newtheorem{corollary}[theorem]{Corollary}

\newtheorem{proposition}[theorem]{Proposition}

\theoremstyle{definition}
\newtheorem{definition}[theorem]{Definition}

\theoremstyle{remark}

\numberwithin{equation}{section}

\usepackage{bbm}
\newcommand{\N}{\mathbbm{N}}
\newcommand{\R}{\mathbbm{R}}
\newcommand{\Z}{\mathbbm{Z}}

\newcommand{\p}{\mathbbm{P}}
\newcommand{\cA}{\mathcal{A}}
\newcommand{\cB}{\mathcal{B}}
\newcommand{\cG}{\mathcal{G}}
\newcommand{\cF}{\mathcal{F}}
\newcommand{\cH}{\mathcal{H}}

\newcommand{\cO}{\mathcal{O}}
\newcommand{\fS}{\mathfrak{S}}
\newcommand{\fT}{\mathfrak{T}}
\newcommand{\E}[1]{\mathbbm{E}\left [ \, #1 \, \right ]}

\renewcommand{\epsilon}{\varepsilon}
\renewcommand{\phi}{\varphi}
\newcommand{\pspace}{(\Omega,\cA,\p)}
\newcommand{\intd}[1]{\,\mathrm{d}#1}
\newcommand{\norm}[1]{\left\lVert #1 \right\rVert}
\newcommand{\scalar}[2]{\left\langle #1,#2 \right\rangle}
\newcommand{\1}[1]{\,\mathbbm{1}\! \left\{ #1 \right\} }

\allowdisplaybreaks
\begin{document}

\title{A Large Deviation Inequality for $\beta$-mixing Time Series and its Applications to the Functional Kernel Regression Model\footnote{This research was supported by the Fraunhofer ITWM, 67663 Kaiserslautern, Germany, which is part of the Fraunhofer Gesellschaft zur F{\"o}rderung der angewandten Forschung e.V.}
}
\author{Johannes T. N. Krebs\footnote{Department of Mathematics, University of Kaiserslautern, 67653 Kaiserslautern, Germany, email: \tt{krebs@mathematik.uni-kl.de} }
}
\date{\today}
\maketitle
\linespread{1.1}
\begin{abstract}
\setlength{\baselineskip}{1.8em}
We give a new large deviation inequality for sums of random variables of the form $Z_k = f(X_k,X_t)$ for $k,t\in \N$, $t$ fixed, where the underlying process $X$ is $\beta$-mixing. The inequality can be used to derive concentration inequalities. We demonstrate its usefulness in the functional kernel regression model of \cite{ferraty_nonparametric_2007} where we study the consistency of dynamic forecasts. \medskip
\end{abstract}

\section{Introduction}~\label{Introduction}

In this article we study asymptotic bounds for the probability 
\begin{align}\label{EqBasicProbability}
\p\left( n^{-1} \left|\sum_{k=1}^n f(X_k,X_t) - \E{ f(X_0,x)} \big|_{x=X_t} \right| \ge \epsilon \right) \text{ where } t \in\Z
\end{align}
for a real-valued function $f$ and for a stationary stochastic process $\{X_k:k\in\Z\}$ which takes values in a state space $(S,\fS)$. If the function $f$ is bounded by $B>1$, we obtain for a certain constant $c$ an exponential rate of convergence for \eqref{EqBasicProbability} which is in $ \cO\left( \exp( - c \epsilon n / (B \log n \log \log n) \,) \right)$. So modulo a logarithmic factor which comes from the dependence in the data, the rate corresponds to the rates of classical large deviations inequalities for independent random variables.\\
Large deviation inequalities are a major tool for the asymptotic analysis in probability theory and statistics. One of the first inequalities of this type was published by \cite{bernstein1927} who considers the case $\p( |S_n| > \epsilon)$, where $S_n=\sum_{k=1}^n X_k$ for bounded real-valued random variables $X_1,\ldots,X_n$ which are i.i.d.\ and have expectation zero. There are various versions and generalizations of Bernstein's inequality, e.g., \cite{hoeffding1963probability}. In particular, deviation inequalities for dependent data such as stochastic processes are nowadays important: Bernstein inequalities for time series are developed in \cite{carbon1983}, \cite{collomb1984proprietes}, \cite{bryc1996large} and \cite{merlevede2009}. \cite{arcones1995bernstein} develops Bernstein-type inequalities for $U$-statistics. \cite{valenzuela2017bernstein} give a further generalization to strong mixing random fields $\{X_s: s\in\Z^N\}$ which are defined on the regular lattice $\Z^N$ for some lattice dimension $N \in\N_+$. A similar version for independent multivariate random variables is given by \cite{ahmad2013probability}. \cite{krebs2017_BernsteinExpGraph} gives a Bernstein inequality for strong mixing random fields which are defined on exponentially growing graphs. The corresponding definitions of dependence and their interaction properties can be found in \cite{doukhan1994mixing} and in \cite{bradley2005basicMixing}. Bernstein inequalities often find their applications when deriving large deviation results or (uniform) consistency statements in nonparametric regression and density estimation, compare for instance \cite{gyorfi}.\\
Consider again Equation~\eqref{EqBasicProbability} where $X$ is a strong mixing process and $t\in\Z$ is a fixed integer. In the special case that the aggregating function $f(x,y) = g(x)h(y)$ is multiplicative, the statement follows from the large deviation inequalities for strong mixing processes given e.g. in \cite{merlevede2009}. However, for a general aggregating function which does not allow a similar decomposition, we face the problem that the process given by $Z_k = f(X_k,X_t)$ is not necessarily strong mixing any more.\\
Thus, we need further assumptions on $X$ which ensure that the contribution of a single summand $f(X_k,X_t)$ is asymptotically negligible in order to derive the necessary bounds. We show that for $\beta$-mixing processes there exist such bounds under certain conditions which are quite likely to be fulfilled in practice.\\
In this manuscript we also give a natural motivation for the new large deviation inequality for~\eqref{EqBasicProbability} and we demonstrate its power in a non-trivial application to the functional kernel regression model. Here the regression operator, which maps from an infinite-dimensional space, e.g., a Hilbert space, to the real numbers, is estimated with a compactly supported kernel function. The literature on this topic is extensive and we refer the reader to \cite{ferraty_nonparametric_2007}, \cite{delsol2009advances} and \cite{ferraty_validity_2010} for a general introduction.\\
This paper is organized as follows: we give the definitions in Section~\ref{Section_DefinitionsNotation}. In Section~\ref{Section_ExponentialInequalities} we present the new exponential inequalities for $\beta$-mixing time series. In Section~\ref{Section_Application} we apply the exponential inequalities in the nonparametric functional regression model of \cite{ferraty_nonparametric_2007} and investigate dynamic forecasts in this model.

\section{Definitions and Notation}\label{Section_DefinitionsNotation}
In this section we give the mathematical definitions and notation which we shall use to derive the results. We first introduce the $\alpha$- and $\beta$-mixing coefficients. The first coefficient is due to \cite{rosenblatt1956central}, the latter was introduced by \cite{kolmogorov1960strong}:
\begin{definition}[$\alpha$- and $\beta$-mixing coefficient]\label{DefMixingCoeff}
Let $\pspace$ be a probability space. Given two sub-$\sigma$-algebras $\cF$ and $\cG$ of $\cA$, the $\alpha$-mixing coefficient is defined by
\begin{align*}
		\alpha(\cF,\cG) \coloneqq \sup \left\{  \left| \p(A\cap B)-\p(A)\p(B)		\right|: A\in\cF, B\in\cG \right\}.
\end{align*}
It is well-known that $\alpha(\cF,\cG) \le 1/4$. The $\beta$-mixing coefficient is defined by
\begin{align*}
		\beta(\cF,\cG) \coloneqq \frac{1}{2} \sup \left\{ \sum_{i\in I} \sum_{j\in J} \big|\p(U_i\cap V_j) - \p(U_i)\p(V_j) \big| \,:\,  (U_i)_{i\in I}, (V_j)_{j\in J} \text{ are finite partitions of } \Omega \right\}.
\end{align*}
Then $2\alpha(\cF,\cG)\le\beta(\cF,\cG) \le 1$, compare to \cite{bradley2005basicMixing}. If $X$ and $Y$ are two random variables on $\pspace$, then we denote by $\alpha(X,Y)$ the mixing coefficient $\alpha(\sigma(X), \sigma(Y) )$ and by $\beta(X,Y)$ the mixing coefficient $\beta( \sigma(X), \sigma(Y) )$. Furthermore, if $(X_k: k\in \Z)$ is a stochastic process, then we write for $n\in\N$
\begin{align*}
		\alpha(n) \coloneqq \sup_{k \in \N} \alpha\big( \sigma( X_s: s\le k ), \sigma( X_s: s\ge k
+n) \big) \text{ and } \beta(n) \coloneqq \sup_{k \in \N} \beta\big( \sigma( X_s: s\le k ), \sigma( X_s: s\ge k
+n) \big).
\end{align*}
The stochastic process $X$ is said to be strong mixing (or $\alpha$-mixing) if $\alpha(n)\rightarrow 0$ ($n\rightarrow \infty$). Furthermore, $X$ is $\beta$-mixing if $\beta(n)\rightarrow0$ ($n\rightarrow \infty$).
The $\beta$-mixing coefficient of two random variables $X,Y$ on $\pspace$ is related to the total variation norm as (cf. \cite{doukhan1994mixing})
\begin{align}\label{BetaMixingTotalVariation}
\beta(X,Y) = \norm{ \p_{X\otimes Y} - \p_X \otimes \p_Y }_{TV} = \sup_{A \in \cA} | \p_{X\otimes Y} (A) - \p_X \otimes \p_Y(A) |.
\end{align}
\end{definition}
Moreover, we denote for a set $A\in\cA$ by $\1{A}$ the indicator function on $A$. If not stated otherwise, we agree on the convention to abbreviate constants by $C$ in the proofs. In the following we derive inequalities of the Bernstein-type for $\beta$-mixing stochastic processes.

\section{Exponential inequalities for \texorpdfstring{$\beta$}{beta}-mixing processes}\label{Section_ExponentialInequalities}

This section contains the main results. The following proposition is the first result of this article: we give inequalities for $\beta$-mixing processes which are similar to Davydov's inequality (cf. \cite{davydov1968convergence}).
\begin{proposition}[$\beta$-mixing and integration w.r.t.\ the joint distribution]\label{BetaMixingIntegrability}
Let $\pspace$ be a probability space and let $(S,\fS)$ and $(T,\fT)$ be measurable, topological spaces. Let $(X,Y) \colon \Omega \rightarrow S\times T$ be $\fS\otimes\fT$-measurable such that the joint distribution of $X$ and $Y$ is absolutely continuous w.r.t.\ their product measure on $\fS\otimes\fT$ with an essentially bounded Radon-Nikod{\'y}m derivative $g$, i.e.,
\[
		\p_{(X,Y)} \ll \p_X \otimes \p_Y \text{ such that } g \coloneqq \frac{ \intd{ \p_{(X,Y)} } }{\intd{ ( \p_{X} \otimes \p_Y ) } } \text{ satisfies } \norm{g}_{\infty, \p_X\otimes \p_Y } < \infty. 
\]
Let $p,q \in [1,\infty]$ be H{\"o}lder conjugate, i.e., $p^{-1} + q^{-1} = 1$. Let $h\colon S\times T \rightarrow \R$ be measurable and $p$-integrable w.r.t.\ the product measure, i.e., $\norm{h}_{p,\p_X \otimes \p_Y} = \E{ \E{ |h(X,y)|^p }|_{y=Y} }^{1/p} < \infty$. Then
\begin{align}\label{EqBetaMixingIntegrability}
	\left| \int_{S\times T} h \intd{ \p_{(X,Y)} } - \int_{S\times T} h \intd{ (\p_X \otimes \p_Y ) } \right| \le 2^{1/q} (1+\norm{g}_{\infty, \p_X \otimes \p_Y} )^{1/p} \norm{h}_{p,\p_X \otimes \p_Y} \beta(X,Y)^{1/q}.
\end{align}
If $p=\infty$ and $q=1$ and if $\norm{h}_{\infty} = \sup_{(x,y)\in S\times T} |h(x,y)|<\infty$, this statement is true without the assumption on the absolute continuity of the distributions:
\begin{align} \label{EqBetaMixingIntegrabilityII}
	\left| \int_{S\times T} h \intd{ \p_{(X,Y)} } - \int_{S\times T} h \intd{ (\p_X \otimes \p_Y )} \right| \le 2 \norm{h}_{\infty} \beta(X,Y).
\end{align}
\end{proposition}
\begin{proof}[Proof of Proposition \ref{BetaMixingIntegrability}]
We begin with the first statement. Therefore we shall use \eqref{BetaMixingTotalVariation} and the following fact: since the joint distribution is absolutely continuous, the total variation (and the $\beta$-mixing coefficient) can be written as
\begin{align}\label{EqEqBetaMixingIntegrability1}
		\beta(X,Y) = \sup_{A \in \cA} | \p_{X\otimes Y} (A) - \p_X \otimes \p_Y(A) | = \frac{1}{2} \int_{S\times T} |g-1| \, \intd{ (\p_X \otimes \p_Y) }.
\end{align}
Indeed, this well-known equality follows if one considers the sets $\{g<1\}, \{g>1\}\in \cA$. Thus, Equation~\eqref{EqEqBetaMixingIntegrability1} and the H{\"o}lder inequality enable us to write
\begin{align*}
		&\left| \int_{S\times T} h \, \intd{ \p_{(X,Y)} } - \int_{S\times T} h \, \intd{	( \p_{X} \otimes \p_Y ) }	\right|  \le  \int_{S\times T} |h| \, |g-1|\, \intd{( \p_{X} \otimes \p_Y ) } \\
		&\le \left( \int_{S\times T} |h|^p \intd{	( \p_{X} \otimes \p_Y ) } \right)^{1/p} \, 2^{1/q}\,  \left( \frac{1}{2} \int_{S\times T} |g-1| \intd{	( \p_{X} \otimes \p_Y )}\right)^{1/q} \, \norm{ g-1 }_{\infty, \p_X \otimes \p_Y }^{(q-1)/q} \\
		&\le 2^{1/q} (1+\norm{g}_{\infty, \p_X \otimes \p_Y })^{1/p} \,\norm{h}_{p, \p_X \otimes \p_Y } \, \beta(X,Y)^{1/q}.
\end{align*}
We come to the second statement. Let $p=\infty$ and $q=1$. Let $h$ be bounded as $\sup_{(x,y)\in S\times T} |h(x,y)|<\infty$ and assume w.l.o.g.\ that $h$ is non negative. Such a function $h$ can be approximated uniformly on $S\times T$ by a sequence of simple functions $h_n$ which converge from below to $h$, i.e., $0 \le \sup_{(x,y)\in S\times T} h(x,y) - h_n(x,y) \le 2^{-n}$. Hence, it suffices to consider simple functions: let $h = \sum_{i=1}^N a_i \1{ A_i }$ for $N\in\N$ and real numbers $a_i\in \R_+$ and pairwise disjoint sets $A_i\in \fS \otimes \fT$. Set $I \coloneqq \{i: 1\le i \le N, \p_X\otimes \p_Y (A_i) \le \p_{(X,Y)} (A_i) \}$ and $J \coloneqq \{1,\ldots,N\} \setminus I$. Hence, $I$ (respectively $J$) is the index sets where the probability of the joint distribution exceeds (respectively is less than) the probability of the product measure. Then, the second statement follows:
\begin{align*}
		&\left| \int_{S\times T} h \, \intd{ \p_{(X,Y)} } - \int_{S\times T} h \, \intd{	( \p_{X} \otimes \p_Y ) }	\right| \\
	& \le \norm{h}_{\infty} \left\{ \sum_{i \in I} \p_{(X,Y)} (A_i) - \p_X\otimes \p_Y (A_i) + \sum_{i\in J}  \p_X\otimes \p_Y (A_i) - \p_{(X,Y)} (A_i) \right\} \\
	&= \norm{h}_{\infty} \left\{ \p_{(X,Y)} \left( \bigcup_{i\in I} A_i \right) -\p_X\otimes \p_Y \left( \bigcup_{i\in I} A_i \right) + \p_X\otimes \p_Y \left( \bigcup_{i\in J} A_i \right) -  \p_{(X,Y)} \left( \bigcup_{i\in J} A_i \right) \right\} \le 2  \norm{h}_{\infty} \beta(X,Y).
\end{align*}
\end{proof}

We present the main results of this article: we derive the exponential inequalities which prove the convergence of the sum $n^{-1} \sum_{k=1}^n f(X_k,X_t) - \E{ f(X_0,x)} \big|_{x= X_t}$ for some $t\in \Z$ fixed. We assume for the rest of the article that the $\beta$-mixing coefficients (and the $\alpha$-mixing coefficients) of $X$ are exponentially decreasing, i.e., there are $\kappa_0, \kappa_1 \in \R_+$ such that $\beta(n) \le  \kappa_0 \exp(-\kappa_1 n)$. \cite{dedecker2004coupling} give an example of such processes: consider a stationary Markov chain $X_n = \psi( X_{n-1} ) + \epsilon_n$ where the innovations are integrable and i.i.d.\ and $\psi$ is Lipschitz-continuous with a Lipschitz constant strictly smaller than 1. Then the chain is geometrically $\beta$-mixing if the distribution of the innovations has an absolutely continuous component which is strictly positive in a neighborhood of zero. For more details, we refer to \cite{dedecker2004coupling} and the references therein.\\
Moreover, we assume w.l.o.g.\ for the next results that the probability space contains an additional random variable $X'_0$ which has the same marginals as the stationary process $X$ but is independent of $X$. We follow the ideas of \cite{merlevede2009} who investigate sums of $\alpha$-mixing processes. Therefore, we define an extension of the discrete process $X=\{ X_k: k\in\Z\}$ to the real interval: define for a real number $t$ the random variable $X_t$ by $X_{\floor{t}}$ which makes the process $\{ X_t: t\in \R\}$ right continuous. In the same way, we extend the mixing coefficients to the real numbers, i.e., $\alpha(t) = \alpha(\floor{t})$ and $\beta(t) = \beta( \floor{t})$ for $t\in\R$. It follows from this definition of the continuous time process that $\sum_{k=1}^n f(X_k,X_t) = \int_{1}^{n+1} f(X_s,X_t) \intd{s}$. We present the main result of this section.

\begin{theorem}\label{LaplaceTransform}
Let $\{X_k: k \in \Z\}$ be a stationary $\beta$-mixing process whose marginals $X_k$ take values in a state space $(S,\fS)$. The $\beta$-mixing coefficients satisfy $\beta(n) \le \kappa_0 \exp(-\kappa_1 n) $ for some $\kappa_0,\kappa_1\in \R_+$. Let $f\colon S\times S \rightarrow [-B, B]$, $B\in \R_+$, be a bounded and measurable function which fulfills $\E{f(X_s,x)} =0$ for $\p_{X_0}$-almost all $x\in S$. Let $A \ge 14 \vee (2\kappa_1)$. Then there is a constant $C$ which only depends on $\kappa_0$ and $\kappa_1$ such that for all $0 < \gamma B \le \frac{1\wedge \kappa_1}{2} \wedge \frac{ \kappa_1}{4 \log A}$
\begin{align*}
		\E{ \exp( \gamma \int_{(0,A]} f(X_k,X_t) ) } \le 3\kappa_0 \exp\left( - \frac{\kappa_1  A}{4 \log A} \right) + \exp \left(	C \gamma^2 B^2 A \log A  + \frac{\gamma B A}{\log A} \right).
\end{align*}
\end{theorem}
\begin{proof}[Proof of Theorem~\ref{LaplaceTransform}]
Let $t\in \R$ be fixed. In the first step we divide the interval $(0,A]$ into three pairwise disjoint intervals $I_L, \tilde{I}$ and $I_R$. Therefore, let $\delta \coloneqq 1/ \log A \in (0,1)$, then $\tilde{I}$ is given by
$$  \tilde{I} \coloneqq (-\delta A/2 + t , \delta A/2 + t] \cap (0,A].$$
$I_L$ is then given as the left residual part of $I$, namely, $I_L = (-\infty, -\delta/2 A +t] \cap (0,A]$. In the same way, $I_R$ is the right residual part, $I_R = (\delta/2 A + t,\infty] \cap (0,A]$. Note that both $I_L$ and $I_R$ have a minimal length of $(1-\delta)A/2$ and that their measure sums up to at most $A$. Furthermore, for all $A \ge 14$, $I_R$ and $I_L$ have a length of at least 4. In order to estimate the Laplace transform, we bound the integral over the interval $\tilde{I}$ by its maximal value and then apply Proposition~\ref{BetaMixingIntegrability} and Proposition~\ref{IbragimovAlphaMixing}:
\begin{align}
		&\E{ \exp\left(\gamma \int_0^A f(X_s,X_t) \intd{s} \right) } \nonumber\\
		&\le \E{ \exp\left(\gamma \int_{I_L \cup I_R} f(X_s,X_t) \intd{s} \right) } \exp\{ \gamma B \delta A \} \nonumber \\
		&\le \Biggl\{ 	\left| \E{ \exp\left(\gamma \int_{I_L \cup I_R} f(X_s,X_t) \intd{s} \right) - \exp\left(\gamma \int_{I_L \cup I_R} f(X_s,X'_0) \intd{s} \right) }		\right|	\nonumber \\
		&\quad + \E{\exp\left(\gamma \int_{I_L \cup I_R} f(X_s,X'_0) \intd{s} \right)} \Biggl\} \, \exp\{ \gamma B \delta A \} \nonumber \\
		&\le \Biggl\{	2 \beta\left( \frac{\delta A}{2} \right) \exp( \gamma B (1-\delta)A ) + \alpha( \delta A) \exp( \gamma B (1-\delta) A ) \nonumber \\
		&\quad + \E{ \E{	\exp\left( \gamma \int_{I_L} f(X_s,x)\intd{s} \right)	}\E{\exp\left( \gamma \int_{I_R} f(X_s,x)\intd{s} \right)			}\Bigg|_{x=X'_0}		}	\Biggl\}\, \exp\{ \gamma B \delta A \} \nonumber \\
		\begin{split}\label{EqLaplace1}
		&\le 3 \kappa_0 \exp\left( - \frac{\kappa_1 \delta A}{2} + \gamma B A \right) \\
		&\quad +  \E{ \E{	\exp\left( \gamma \int_{I_L} f(X_s,x)\intd{s} \right)	}\E{\exp\left( \gamma \int_{I_R} f(X_s,x)\intd{s} \right)			}\Bigg|_{x=X'_0}		}	\, \exp\{ \gamma B \delta A \}.
		\end{split}
\end{align}
Consider the inner expectations in \eqref{EqLaplace1}: since the random variable $f(X_s,x)$ is centered for $\p_{X_0}$-almost all $x\in S$, we can apply Lemma 10 from \cite{merlevede2009} and obtain if $(1-\delta)A/2 > 4 \vee (2\kappa_1)$ that
\begin{align}
		\E{	\exp\left( \gamma \int_{I_L} f(X_s,x)\intd{s} \right)	}\E{\exp\left( \gamma \int_{I_R} f(X_s,x)\intd{s} \right)			} &\le \exp\left(	C \gamma^2 B^2 ( |I_L| \log |I_L| + |I_R| \log |I_R|)	\right) \nonumber \\
		&\le \exp \left(	C \gamma^2 B^2 A \log A \right), \label{EqLaplace2}
\end{align}
where $0 \le \gamma B \le (\kappa_1 \wedge 1)/2$ and $C$ is a constant which only depends on the bound of the mixing coefficients $\kappa_0,\kappa_1$. Combining \eqref{EqLaplace1} with \eqref{EqLaplace2} yields
$$ \E{\exp\left(\gamma \int_0^A f(X_s,X_t) \intd{s} \right) } \le  3\kappa_0 \exp\left( - \frac{\kappa_1  A}{2 \log A} + \gamma B A \right) + \exp \left(	C \gamma^2 B^2 A \log A  + \frac{\gamma B A}{\log A} \right). $$
\end{proof}

Since the bound $B$ is the only property of the function $f$ which is determining the bound on the Laplace transform given in Theorem~\ref{LaplaceTransform}, we can easily extend the statement to a sequence of functions $(f_n: n:\N)$ which are all uniformly bounded: $f_n\colon S\times S \rightarrow [-B,B]$ for $n\in\N$. Therefore, we give the corollary

\begin{corollary}\label{LargeDeviation}
Let the stochastic process $X$ be as in Theorem~\ref{LaplaceTransform}. Let $(f_n: n:\N)$ be a sequence of functions such that each element $f_n\colon S\times S \rightarrow [-B,B]$ is as in Theorem~\ref{LaplaceTransform} and satisfies in particular $\E{ f_n(X_0,x)}=0$ for $\p_{X_0}$-almost all $x\in S$. Then for all $n\in \N$ there are constants $a_1,a_2$ which only depend on the parameter $\kappa_0, \kappa_1$ such that
\begin{align*}
		\p\left(	n^{-1} \left| \sum_{k=1}^n f_n(X_k,X_t) \right| \ge \epsilon	\right) \le a_1 \exp\left( - a_2 \frac{ \epsilon \, n }{ B \log n \log \log n } \right).
\end{align*}
\end{corollary}
\begin{proof}[Proof of Corollary~\ref{LargeDeviation}]
If we apply Theorem~\ref{LaplaceTransform} to the interval $(0, n]$ and to the process $X$ which is formally shifted one integer to the left, we find with the choice $\gamma = (\kappa_1 \wedge 1) / (4 B \log n \log \log n)$ and Markov's inequality that for $n\ge 14$
\begin{align*}
		\p\left(	n^{-1} \left| \sum_{k=1}^n f_n(X_k,X_t) \right| \ge \epsilon	\right) &\le \exp\left(- \frac{ (\kappa_1 \wedge 1)  \epsilon \, n }{ 4 B \log n \log \log n } \right) \\
		&\quad \cdot \Biggl\{ 3\kappa_0 \exp\left( - \frac{\kappa_1  n}{4 \log n} \right) +\exp \left(	C \gamma^2 B^2 n \log n  + \frac{\gamma B n}{\log n} \right) \Biggl\} .
\end{align*}
Note that the last term in the curly brackets grows at a rate of $n/ \left(\log n (\log \log n)^2 \right)$. This finishes the proof.
\end{proof}

We consider the special case where $f_n(X,Y) \equiv X$ in the context of Corollary~\ref{LargeDeviation}. We obtain in this case for a $\beta$-mixing process $X$ that $\p\left(	\left| n^{-1} \sum_{k=1}^n X_k \right| \ge \epsilon	\right) \le a_1 \exp\left\{	- a_2 \epsilon\, n/ (\log n \log \log n) \right\}$. This corresponds to the rate given by \cite{merlevede2009} who investigate concentration inequalities of centered, real-valued and strongly mixing processes $X$ where the $\alpha$-mixing coefficients decay at an exponential rate. This means that our new inequality from Corollary~\ref{LargeDeviation} attains the same rate in this special case.\\
We can derive an extension of Theorem~\ref{LaplaceTransform} to a sequence of unbounded functions $f_n\colon S\times S \rightarrow \R$.

\begin{theorem}\label{UnboundedExponentialInequality}
Let the stationary process $X$ be as in Theorem~\ref{LaplaceTransform}. Furthermore, let $(f_n:n\in\N)$ be a sequence of functions $f_n\colon S\times S \rightarrow \R$ which fulfills $\E{ f_n(X_0,x)}=0$ for $\p_{X_0}$-almost all $x\in S$. Let the joint distribution of $(X_0,X_k)$ be absolutely continuous w.r.t.\ the product measure such that the corresponding Radon-Nikod{\'y}m derivatives are essentially bounded uniformly over $k\in \N$, i.e.,
$$ \sup_{ k\in\N } \norm{ \frac{ \intd\p_{ (X_0,X_k) } }{\intd{ \left(\p_{X_0} \otimes \p_{X_k}  \right)}} } < \infty. $$
Then, there are constants $a_1,a_2 \in\R_+$ which do not depend on $n \in \N$ and $t\in \Z$ such that for all $\epsilon > 0$
\begin{align}\begin{split}\label{EqUnboundedExponentialInequality0}
&\p\left(	n^{-1} \left| \sum_{k=1}^n f_n(X_k,X_t) \right| \ge \epsilon	\right) \\
 &\le \inf_{B>1} \Bigg\{ a_1 \epsilon ^{-1} \exp\left( -a_2 \frac{\epsilon n}{B\log n \log \log n}  \right) + 4 \epsilon^{-1} (k-1)^{-1} B^{-(k-1)} \E{ f_n(X_0,X'_0)^k} \\
&\quad + a_1 (n\epsilon)^{-1} \E{ f_n(X_0,X'_0)^{p r}}^{1/(pr)} B^{-k/(pu)} \E{ f_n(X_0,X'_0)^k}^{1/(pu)} \Bigg\}
\end{split}\end{align}
where $p>1$, $r^{-1}+u^{-1}=1, r,u>1$ and $k>1$.
\end{theorem}
\begin{proof}[Proof of Theorem~\ref{UnboundedExponentialInequality}]
If the right hand side of \eqref{EqUnboundedExponentialInequality0} is infinite, there is nothing to prove. So we can assume that the parameters $k,p,r$ and $u$ are such that the all moments on the right hand side exist. Since $\E{ f_n(X_0,x)} = 0$ for $\p_{X_0}$-almost all $x\in S$, we use for $B>1$ the fundamental decomposition $f_n(x,y) = f_{n,+} (x,y) + f_{n,0}(x,y) + f_{n,-} (x,y)$ where
\begin{align*}
		f_{n,+}(x,y) \coloneqq f_n(x,y) - \min( f_n(x,y), B) \ge 0, \quad f_{n,-}(x,y) &\coloneqq f_n(x,y) - \max(f_n(x,y), -B) \le 0 \\
		&\text{ and } f_{n,0}(x,y) \coloneqq \max( \min(f_n(x,y), B), -B).
\end{align*}
Then,
\begin{align}\begin{split}\label{UnboundedExponentialInequality1}
		\p\left( n^{-1} \left|\sum_{k=1}^n f_n(X_k,X_t) \right| > 3 \epsilon		\right) &\le \p\left( n^{-1} \left|\sum_{k=1}^n f_{n,+}(X_k,X_t) - \E{f_{n,+}(X_k,x)} |_{x=X_t} \right| >  \epsilon		\right)  \\
		&\quad + \p\left( n^{-1} \left|\sum_{k=1}^n f_{n,0}(X_k,X_t) - \E{f_{n,0}(X_k,x)} |_{x=X_t} \right| >  \epsilon		\right) \\
		&\quad + \p\left( n^{-1} \left|\sum_{k=1}^n f_{n,-}(X_k,X_t) - \E{f_{n,-}(X_k,x)} |_{x=X_t} \right| >  \epsilon		\right) .
		\end{split}
\end{align}
The asymptotic behavior of the second probability in \eqref{UnboundedExponentialInequality1} is given in Corollary~\ref{LargeDeviation}. The first and the third probability can be bounded with Markov's inequality and Proposition~\ref{BetaMixingIntegrability}, we only consider $f_{n,+}$ here:
\begin{align}
	&\p\left( n^{-1} \left|\sum_{k=1}^n f_{n,+}(X_k,X_t) - \E{f_{n,+}(X_0,x)} |_{x=X_t} \right| >  \epsilon		\right) \nonumber \\
	&\le (n \epsilon)^{-1} \sum_{k=1}^n \left| \E{ \left| f_{n,+}(X_k,X_t) - \E{f_{n,+}(X_0,x)} |_{x=X_t} \right| - \left| f_{n,+}(X_k,X'_0) - \E{f_{n,+}(X_0,x)} |_{x=X'_0} \right|	} \right| \nonumber \\
	&\quad + \epsilon^{-1} \E{ \left| f_{n,+}(X_0,X'_0) - \E{f_{n,+}(X_0,x)} |_{x=X'_0} \right|	} \nonumber \\
	&\le C (n\epsilon)^{-1} \sum_{k=1}^n \beta( |t-k| )^{1/q} \, \E{ \left| f_{n,+}(X_0,X'_0) - \E{f_{n,+}(X_0,x)} |_{x=X'_0} \right|^{p}	} ^{1/p} \nonumber \\
	&\quad + \epsilon^{-1} \E{ \left| f_{n,+}(X_0,X'_0) - \E{f_{n,+}(X_0,x)} |_{x=X'_0} \right|	}	\nonumber \\
	\begin{split}\label{UnboundedExponentialInequality2}
	&\le C	(n \epsilon)^{-1} \E{ |f_{n,+}(X_0,X'_0) |^p }^{1/p} + 2\epsilon^{-1} \E{ \left| f_{n,+}(X_0,X'_0) \right|},
	\end{split}
\end{align}
where $p,q > 1$ such that $p^{-1} + q^{-1} = 1$ as in Equation~\eqref{EqBetaMixingIntegrability} and where we use in the last inequality that the $\beta$-mixing coefficients are summable. We bound the two expectations given in \eqref{UnboundedExponentialInequality2} further. For the first expectation we obtain for $k>1$
\begin{align*}
		\E{ f_{n,+}(X_0,X'_0) } &= \E{ \int_B^{\infty} \p(f_n(X_0,y)>t)\intd{t} \Big|_{y=X'_0} } \le \E{ \int_B^{\infty} t^{-k} \intd{t} \; \E{f_n(X_0,y)^k} \Big|_{y=X'_0} } \\
		&\le (k-1)^{-1} B^{-(k-1)} \, \E{f_n(X_0,X'_0)^k} .
\end{align*}
The second expectation can be bounded with H{\"o}lder's inequality for $r^{-1}+u^{-1}=1,r,u>1$ and Markov's inequality:
\begin{align*}
		\E{ |f_{n,+}(X_0,X'_0) |^p } &= \E{ |f_n(X_0,X'_0)-B |^{pr} \1{f_n(X_0,X'_0)\ge B} }^{1/r} \p(f_n(X_0,X'_0)\ge B)^{1/u} \\
		&\le \E{ f_n(X_0,X'_0)^{pr} }^{1/r} B^{-k/u} \E{f_n(X_0,X'_0)^k}^{1/u},
\end{align*}
for $k>1$. This proves the claim.
\end{proof}

\section{An application in the nonparametric functional regression model}\label{Section_Application}
In this section, we embed the developed inequalities in the nonparametric functional kernel regression model of \cite{ferraty_nonparametric_2007}, \cite{delsol2009advances} and \cite{ferraty_validity_2010}. We cannot discuss all details of this model here due to its complexity and assume that the reader has some prior knowledge on this subject. Nevertheless, we describe all necessary assumptions: let $(\cH,\scalar{\cdot}{\cdot})$ be a separable Hilbert space over $\R$ where the norm $\norm{\,\cdot\,}$ is induced by the inner product $\scalar{\cdot}{\cdot}$. In typical applications the Hilbert space is given as a function space $L^2([0,1],\cB([0,1]),\intd{x})$ or more general $L^2(K,\cB(K),\nu)$ for a bounded and convex set $K\in\R^d$ and a finite measure $\nu$. Let $\{ (X_k,Y_k): k\in \Z\} \subseteq \cH \times \R$ be a stationary process which is $\beta$-mixing with exponentially decaying coefficients. The process fulfills the regression equation
\begin{align}\label{EqRegression1}
	Y_k = \psi (X_k) + \epsilon_k, \quad k\in\Z.
	\end{align}
The regression function $\psi\colon \cH \rightarrow \R$ is Lipschitz-continuous, i.e., $| \psi(x) - \psi(y) | \le L \norm{ x-y}$ for all $x,y \in \cH$. The error terms fulfill $\E{ \epsilon_k | X_k} =0$ and $\E{\epsilon_k^2| X_k}< \infty$. Given an observed sample $(X_1,Y_1),\ldots,(X_n,Y_n)$ the estimator of $\psi$ is
\begin{align}\label{EqRegression2}
	\hat{\psi} (x) = \frac{ \sum_{k=1}^n Y_k K( \norm{ X_k-x}/h) }{\sum_{k=1}^n  K( \norm{ X_k-x}/h) },
	\end{align}
where $h>0$ is the bandwidth and $K$ is a kernel function. The kernel function is supported in $[0,1]$ and zero otherwise. It admits a continuous derivative $K'$ on $[0,1)$ such that $K'\le 0$ and $K(1) >0$ as in \cite{ferraty_nonparametric_2007}. Moreover, we define the quantities
\begin{align*}
		 \hat{g}_h (x) \coloneqq (n F_x(h) ) ^{-1} \sum_{k=1}^n Y_k K( \norm{ X_k-x}/h) \quad \text{ and }\quad
		\hat{f}_h (x) \coloneqq  (n F_x(h) ) ^{-1} \sum_{k=1}^n  K( \norm{ X_k-x}/h) ,
\end{align*}
where $F_x(h) = \p( \norm{ X_0 - x} \le h)$. Note that we do not normalize $\hat{f}_h$ and $\hat{g}_h$ by a division with the bandwidth instead we multiply the inverse of the small ball probability $F_x(h)$.  Here we assume that $F_x(h) > 0$ for $h>0$ and that $F_x(0) =0$ for all $x\in \cH$. We choose the bandwidth $h$ as a function of $n$ such that for $h\rightarrow 0$ as $n\rightarrow \infty$ the summability condition
\begin{align}\label{EqRegression3}
		\sum_{n=1}^{\infty} n^{-2} \E{ F_{X_0}(h)^{-2} } (\log n)^2 (\log \log n)^2 < \infty
\end{align}
is fulfilled where we assume that $\E{ F_{X_0}(h)^{-2} }$ is finite for each $h>0$. E.g. we can choose $h\rightarrow 0$ such that $n^{-2} \E{ F_{X_0}(h)^{-2} } = o\left(n^{-(1+\delta)} \right)$ for some $\delta\in(0,1)$. Note that for the pointwise convergence $\hat{\psi}(x) \rightarrow \psi(x)$ \cite{ferraty_nonparametric_2007} require $n^{-1} F_x(h)^{-1} \rightarrow 0$.\\
Assume that there is a small ball probability function $\tau: [0,1] \rightarrow [0,1]$ which the uniform limit of the fractions $F_x(hs)/F_x(h)$, i.e.,
\begin{align}\label{EqSmallBall}
	\lim_{h\downarrow 0} \sup_{x\in \cH} \left| \frac{ F_x(hs)}{F_x(h)} - \tau(s)		\right| = 0.
 \end{align}
Further, set $	M = K(1) - \int_0^1 K'(s) \tau (s) \intd{s}	$ which is assumed to be positive.\\
We shortly discuss the issue of the convergence of $F_x(hs)/F_x(h)$ for $h \downarrow 0$ and $s\in(0,1)$, a more detailed introduction to this topic offer \cite{ferraty_nonparametric_2007} and the references therein. If the underlying Hilbert space is a function space, then one has in many applications that for a point $x$ in the Hilbert space $\p( \norm{ X_0 - x} < h ) \sim C(x) \p( \norm{X_0} < h)$ for $h \downarrow 0$. Hence the small ball problem at the point $x$ shifts to the origin and in the limit the quotient $F_x(hs)/F_x(h)$ becomes independent of the point $x$. This motivates assumption \eqref{EqSmallBall}.\\
We study the dynamic forecast $\hat{\psi} (X_t)$ for an observation $X_t$, this means, we give sufficient conditions that the difference $\hat{\psi}(X_t)-\psi(X_t)$ converges to zero $a.s.$ We can use Theorem~\ref{LaplaceTransform} and Theorem~\ref{UnboundedExponentialInequality} for this problem and do not need to consider the difference $\hat{\psi}-\psi$ pointwise for $x\in\cH$. We have the following theorem

\begin{theorem}\label{DynamicForecast}
Let $(X,\epsilon) = \{(X_k,\epsilon_k):k\in\Z\} \subseteq \cH \times \R$ be a stochastic process as in Theorem~\ref{UnboundedExponentialInequality}. Let $\{\epsilon_k: k \in \Z\}$ be a sequence of innovations such that the regression model~\eqref{EqRegression1} is fulfilled. Let \eqref{EqRegression3} be satisfied for $h\rightarrow 0$ as $n\rightarrow\infty$. 
Then for a component $X_t$ of the process ($t\in\Z$) it is true that $\hat{\psi}(X_t) \rightarrow \psi(X_t)$ a.s as $n\rightarrow \infty$.
\end{theorem}
\begin{proof}[Proof of Theorem~\ref{DynamicForecast}]
We show that under the conditions we have both $\hat{f}_h (X_t) \rightarrow M$ $a.s.$ and $\hat{g}_h(X_t) \rightarrow \psi(X_t) M$ $a.s.$ We begin with $\hat{f}_h(X_t)$. We have that $\left|\hat{f}_h (X_t) - \E{ \hat{f}_h (x) }\Big|_{x=X_t} \right| \rightarrow 0$ $a.s.$ Indeed, let $\epsilon > 0$ be arbitrary but fixed, we apply Equation~\eqref{EqUnboundedExponentialInequality0} from Theorem~\ref{UnboundedExponentialInequality} with the following parameters
$$
		B = \frac{n}{ (\log n)^2 \, (\log \log n)^2 },\quad p=3/2,\quad k=3 \quad \text{and}\quad u=r=2.
$$
Note that $f_n(X_0,X'_0)$ corresponds to the function $K\left(\norm{X_0-X'_0}/h \right)/F_{X'_0}(h)$ in this case. An application of the theorem of Fubini-Tonelli yields
\begin{align*}
		\E{ f_n(X_0,X'_0)^{k} } &= \E{ \left(\frac{ K\left(\norm{X_0-X'_0}/h\right)}{F_{X'_0}(h)}\right)^{3} } \\
		&= \E{ \left(F_{X'_0}(h)\right)^{-2} \cdot \left( K(1)^3 - \int_0^1 (K(s)^3)' \frac{ F_{X'_0}(hs)}{F_{X'_0}(h)} \intd{s} \right) } = O\left( \E{ F_{X'_0}(h)^{-2} } \right)
\end{align*}
by the uniform convergence of the small probability from \eqref{EqSmallBall}. Hence, we obtain for some constants $a_1,a_2\in\R_+$ that
\begin{align}\begin{split}\label{EqRegression4}
		\p\left( \left| \hat{f}_h(X_t) - \E{ \hat{f}_h (x) } \Big|_{x=X_t} \right| \ge \epsilon \right) &\le  a_1 \epsilon^{-1} \exp\left( -a_2 \log n \cdot \log\log n		\right) \\
		&\quad+  a_1\epsilon^{-1} (\log n)^2 (\log \log n)^2 n^{-2} \E{F_{X'_0}(h)^{-2} }.
\end{split}\end{align}
Consequently, the probabilities from the left hand side of \eqref{EqRegression4} are summable over $n\in\N_+$ for $\epsilon>0$ fixed because of the choice of the bandwidth from \eqref{EqRegression3}. Thus, the claim that $\hat{f}_h(X_t) - \E{\hat{f}_h (x)}\Big|_{x=X_t} \rightarrow 0$ $a.s.$ follows with an application of the first Borel-Cantelli Lemma. An application of the theorem of Fubini-Tonelli yields
\begin{align}
		\sup_{x\in \cH} \left| \E{ \hat{f}_h (x) } - M \right| &= \sup_{x\in \cH} \left| \E{ \frac{K( \norm{ X_0 - x}/h )}{ F_x(h)} } - M \right| \nonumber \\
		&\le \sup_{x\in \cH} \int_0^1 K'(s) \left| \frac{ F_x(hs)}{F_x(h)} - \tau(s)		\right| \intd{s} \rightarrow 0. \nonumber
\end{align}
The last inequality follows similarly as in \cite{ferraty_nonparametric_2007} and from the requirement \eqref{EqSmallBall}. This proves that $\hat{f}_h(X_t) \rightarrow M>0$ $a.s.$ Consider next $\hat{g}_h(X_t)$. Once more, we have $\left|\hat{g}_h (X_t) - \E{ \hat{g}_h (x) }\Big|_{x=X_t}  \right| \rightarrow 0$ $a.s.$ using the requirement \eqref{EqRegression3}. Furthermore, we obtain for a point $x\in\cH$ with the assumption that the regression function $\psi$ is Lipschitz continuous and that the conditional expectation of the innovations is zero
\begin{align}
		\left| \E{ \hat{g}_h (x) }- M \psi(x) \right| &\le \left| \E{ \left( \psi(X_0)-\psi(x) \right) \frac{K( \norm{ X_0 - x}/h )}{ F_x(h)} } \right| + |\psi(x)|\left| \E{ \frac{K( \norm{ X_0 - x}/h )}{ F_x(h)} } - M \right| \nonumber \\
		&\le Lh  M + (|\psi(x) |+o(1)) \left| \E{ \frac{K( \norm{ X_0 - x}/h )}{ F_x(h)} }- M \right|. \nonumber 
\end{align}
This ensures that $\left|\E{\hat{g}_h (x) }\big|_{x=X_t} - M \psi(X_t)\right|\rightarrow 0$ $a.s.$ and proves the second statement $\hat{g}_h(X_t) \rightarrow \psi(X_t) M$ $a.s.$\\
All in all, we have $\hat{\psi} (X_t) - \psi(X_t) = \hat{g}_h(X_t) /\hat{f}_h(X_t) - \psi(X_t)  \rightarrow  0$ $a.s.$ by the continuous mapping theorem.
\end{proof}

\appendix
\section{Appendix}
The following statement can be seen as a multivariate generalization of \cite{davydov1968convergence} in a special case:
\begin{proposition}[\cite{ibragimov1962some}]\label{IbragimovAlphaMixing}
Let $Z_1,\ldots,Z_n$ be real-valued non-negative random variables each $a.s.$ bounded. Denote by $\alpha \coloneqq \sup_{k\in \{1,\ldots,n\} } \alpha\left( \sigma( Z_i: i \le k), \sigma( Z_i: i > k) \right)$. Then
$$ \left| \E{ \prod_{i=1}^n Z_i } - \prod_{i=1}^n \E{Z_i} \right| \le (n-1) \, \alpha\, \prod_{i=1}^n \norm{Z_i}_{\infty}.$$
\end{proposition}

\section*{Acknowledgments}
The author is very grateful to a referee for thoughtful suggestions and comments which significantly improved and clarified the manuscript.\\

\end{document}